\documentclass[a4paper,reqno,11pt]{amsart}

\usepackage{amsmath, amsfonts, amssymb, amsthm, amscd}
\usepackage{graphicx}
\usepackage{psfrag}
\usepackage{perpage}
\usepackage{url}
\usepackage{color}

\usepackage[utf8]{inputenc}
\usepackage[T1]{fontenc}
\usepackage{microtype}

\usepackage[a4paper,scale={0.72,0.74},marginratio={1:1},footskip=7mm,headsep=10mm]{geometry}

\usepackage{hyperref}

\numberwithin{equation}{section}

\newtheorem{theorem}{Theorem}[section]
\newtheorem{lemma}[theorem]{Lemma}
\newtheorem{proposition}[theorem]{Proposition}

\newtheorem{remark}[theorem]{Remark}

\DeclareMathSymbol{\leqslant}{\mathalpha}{AMSa}{"36} % nicer `smaller or equal'
\DeclareMathSymbol{\geqslant}{\mathalpha}{AMSa}{"3E} % nicer `larger or equal'
\DeclareMathSymbol{\eset}{\mathalpha}{AMSb}{"3F}     % nicer `emptyset'
\newcommand{\dd}{\text{\rm d}}             % a straight d for differentials

\newcommand{\R}{\mathbb{R}}
\newcommand{\C}{\mathbb{C}}
\newcommand{\Z}{\mathbb{Z}}
\newcommand{\N}{\mathbb{N}}

\newcommand{\PEfont}{\mathrm}

\DeclareMathOperator{\p}{\ensuremath{\PEfont P}}
\DeclareMathOperator{\e}{\ensuremath{\PEfont E}}

\newcommand{\ind}{{\sf 1}}

\renewcommand{\epsilon}{\varepsilon} 
\renewcommand{\theta}{\vartheta} 
\renewcommand{\rho}{\varrho} 
\renewcommand{\phi}{\varphi}

\newenvironment{myenumerate}{%
\renewcommand{\theenumi}{\arabic{enumi}}%
\renewcommand{\labelenumi}{{\rm(\theenumi)}}%
\begin{list}{\labelenumi}
	{%
	\setlength{\itemsep}{0.4em}%
	\setlength{\topsep}{0.5em}%
	\setlength\leftmargin{2.45em}%
	\setlength\labelwidth{2.05em}%
	\setlength{\labelsep}{0.4em}%
	\usecounter{enumi}%
	}%
	}%
{\end{list}
}

\renewenvironment{enumerate}{
\begin{myenumerate}}%
{\end{myenumerate}}

\newenvironment{myitemize}{%
\begin{list}{$\bullet$}% 
 	{%
	\setlength{\itemsep}{0.4em}%
	\setlength{\topsep}{0.5em}%
	\setlength\leftmargin{2.45em}%
	\setlength\labelwidth{2.05em}%
	\setlength{\labelsep}{0.4em}%
	}%
	}%
{\end{list}}

{\end{myitemize}}

\MakePerPage[2]{footnote} % restarts footnote counter at every new page

\title{A note on directly Riemann integrable functions}

\author{Francesco Caravenna}

\address{Dipartimento di Matematica e Applicazioni, Universit\`a
degli Studi di Milano-Bicocca, via Cozzi 53, 20125 Milano, Italy}

\email{francesco.caravenna@unimib.it}

\keywords{Directly Riemann Integrable Function, Renewal Theory,
Renewal Density Theorem, Local Limit Theorem, Convolution.}

\subjclass[2010]{60K05; 26A42; 44A35}
\date{\today}

\def\dd{\mathrm{d}}

\begin{document}

\begin{abstract}
A non-negative function $f$, defined on the real line or on a half-line, is said
to be \emph{directly Riemann integrable} (d.R.i.) if the upper and lower Riemann 
sums of $f$ over the whole (unbounded) domain converge to the same finite limit, as the 
mesh of the partition vanishes.
In this note we show that, for a Lebesgue-integrable function $f$,
very mild conditions are enough to ensure that some $n$-fold convolution of $f$ 
with itself is d.R.i..
Applications to renewal theory and to local limit theorems are discussed.
\end{abstract}

\maketitle

\section{Introduction}

\subsection{The problem}
A non-negative function $g: \R \to [0,\infty)$ is said to be
\emph{directly Riemann integrable} (d.R.i.) if its upper and lower Riemann sums
\emph{over the whole real line} converge to the same finite limit, as the mesh of the partition
vanishes:
\begin{equation} \label{eq:dRi}
	\lim_{\delta \downarrow 0} \
	\sum_{m\in\Z} \delta \, \bigg(  \sup_{z \in [m\delta, (m+1)\delta)} g(z) \bigg)
	\,=\, \lim_{\delta \downarrow 0} \
	\sum_{m\in\Z} \delta \, \bigg(  \inf_{z \in [m\delta, (m+1)\delta)} g(z) \bigg)
	\,\in\, (-\infty,+\infty) \,.
\end{equation}
If the function $g$ may also take negative values, it is said to be d.R.i.
if both its positive and negative parts $g^+$ and $g^-$ are so.
We refer to \cite[\S V.4]{cf:Asm} and \cite[\S XI.1]{cf:Fel2} for more details.

%Intuitively, a d.R.i. function can be efficiently approximated by step
%functions with intervals of a fixed length.

%in particular, a d.R.i. function is bounded and continuous a.e. with respect 
%to Lebesgue measure and, of course, the limit in \eqref{eq:dRi} coincides
%with $\int_\R g(x) \, \dd x$.

\smallskip

Every d.R.i. function is necessarily in $L^1(\R, Leb)$ and vanishes at infinity,
but the converse might fail, even for continuous functions, because of the
possible oscillations at infinity.
The aim of this note is to show that very mild conditions on
$f \in L^1(\R,Leb)$ are
enough to ensure that some convolution of $f$ with itself is d.R.i.,
cf. Theorem~\ref{th:main} below.

Beyond its intrinsic interest,
our main motivations for such a result come from local limit theorems and
renewal theory, where d.R.i. functions play an important role.
In particular, we suggest to keep in mind the special case when
$f$ is a probability density function on $\R$.

\subsection{The main result}

Given a Lebesgue-integrable function $f \in L^1(\R, Leb)$,
let us denote by $f_k(\cdot) = f^{*k}(\cdot)$ the $k$-fold convolution
of $f$ with itself, that is
\begin{equation} \label{eq:conv}
	f_1(x) \,:=\, f(x)\,, \qquad
	f_{k+1}(x) \,:=\, (f_{k} * f)(x) \,=\,
	\int_\R f_{k}(x-y) \, f(y) \, \dd y \,, \quad \forall k \in \N\,.
\end{equation}
Our main result, that we prove in section~\ref{sec:proof}, reads as follows.

\begin{theorem}\label{th:main}
Let $f \in L^1(\R,Leb)$ satisfy the following assumptions:
\begin{enumerate}
\item\label{it:1} $f_{k_0} \in L^\infty(\R,Leb)$ for some $k_0\in\N$;

\item\label{it:2} $\int_\R |x|^\epsilon |f(x)| \, \dd x \,<\, \infty$ for some $\epsilon > 0$.
\end{enumerate}
Then there exists $k_1 \in \N$ such that
for every $k\ge k_1$ the function $x \mapsto (1+|x|^\epsilon) f_k(x)$
is bounded, continuous and directly Riemann integrable.
In particular, $f_k$ itself is bounded, continuous and 
directly Riemann integrable, for every $k\ge k_1$.
\end{theorem}

\subsection{Organization of the paper}
The rest of the introduction is devoted to discussing
the role of the two assumptions of Theorem~\ref{th:main}.
In section~\ref{sec:apprenth} we present some
applications of Theorem~\ref{th:main} to renewal theory
(and, more generally, to local limit theorems, cf. \S\ref{sec:appllt}).
We mention in particular Proposition~\ref{th:cor},
which provides a local version of the renewal theorem for
heavy-tailed renewal processes.
The proof of Theorem~\ref{th:main} is contained in section~\ref{sec:proof},
while some technical points are deferred to the appendix.

Henceforth we write $L^p :=
L^p(\R,Leb)$ for short.

\subsection{Discussion}
\label{sec:roleass}

%Let us discuss assumption \eqref{it:1} in Theorem~\ref{th:main}.
A d.R.i. function is necessarily bounded, since otherwise every
upper or lower Riemann sum
would be infinite. Therefore \emph{assumption \eqref{it:1} is necessary 
for Theorem~\ref{th:main} to hold}. Let us now give a standard and more concrete
reformulation of this assumption in terms of the
Fourier transform $\widehat f(\theta) := \int_\R e^{i\theta x} f(x) \, \dd x$ of $f$.

\begin{lemma} \label{th:necsuf}
A function $f \in L^1$ satisfies assumption \eqref{it:1} of Theorem~\ref{th:main}
%that is $f_{k_0} \in L^\infty$ for some $k_0 \in \N$,
if and only if $\widehat f \in L^p$ for some $p \in [1,\infty)$.
\end{lemma}

\noindent
From this, we can deduce a very practical
sufficient condition.

\begin{lemma} \label{th:suf}
Assumption \eqref{it:1} of Theorem~\ref{th:main} is satisfied
if $f \in L^1 \cap L^p$, for some $p \in (1,\infty]$.
\end{lemma}

\noindent
The (quite standard) proofs of these two lemmas are given
in \S\ref{sec:proofnecsuf} and~\S\ref{sec:proofsuf} below.

\begin{remark}\rm
For a function $f \in L^1$, the condition $f \in L^p$ for some $p>1$
is very mild and shows that assumption~\eqref{it:1}
is typically verified in concrete situations.
Let us mention, however, that (somewhat pathological) examples of functions
$f \in L^1$ not satisfying assumption~\eqref{it:1} do exist, cf.
Examples (a), (b), (c) in~\cite[\S XV.5]{cf:Fel2}.
\end{remark}

\begin{remark}\rm
When $f$ is a probability density with zero mean and unit variance,
assumption \eqref{it:1} is a necessary and sufficient condition for the
Local Central Limit Theorem, that is, for
%, for the uniform convergence 
$\lim_{n\to\infty}
\sup_{x\in\R} |\sqrt{n} f_n(\sqrt{n} x) - \frac{1}{\sqrt{2\pi}} e^{-x^2/2}| = 0$,
cf. \cite[Theorem~2]{cf:Fel2} and the following lines.
\end{remark}

%\smallskip

Let us now discuss assumption~\eqref{it:2}.
This is \emph{not} necessary for Theorem~\ref{th:main} to hold, as the following
example shows: the function $f(x) := \frac{1}{x \, (\log x)^2} \ind_{[e,\infty)}(x)$ is
in $L^1$, it does not satisfy assumption~\eqref{it:2}, for any $\epsilon > 0$,
but $f$ is d.R.i.. Only the latter statement requires a proof.
Being Riemann integrable on every compact interval
(it is continuous except at the point $e$), it suffices to verify
that an upper Riemann sum of $f$ is finite, by Lemma~\ref{th:reformulation} below.
Since
$f(x) = 0$ for $x < e$ and $f$ is \emph{decreasing} on $[e,\infty)$, we have
\begin{equation*}
	\sum_{m\in\Z} \, \sup_{z \in [m,m+1)} f(z)
	\,=\, \sum_{m\in\N, m \ge 2} \, \sup_{z \in [m,m+1)} f(z)
	\,\le\, \frac{2}{e} \,+\, \int_3^\infty f(z) \, \dd z \,<\, \infty \,,
\end{equation*}
where $\frac{2}{e}$ accounts for the terms $m=2,3$ of the sum,
while for $m \ge 4$ we have used the estimate $\sup_{z \in [m,m+1)} f(z)
\le  \int_{m-1}^m f(z) \, \dd z$, by monotonicity of $f$.

\begin{remark}\rm
The same argument shows that if some convolution $f_k$
is bounded, continuous and dominated in absolute value by $g \in L^1$,
with $g$ non-increasing in a neighborhood of infinity, then $f_k$
is d.R.i.. Such a condition, however,
seems difficult to check in terms of $f$.
\end{remark}

\begin{remark}\rm
Although not necessary, assumption~\eqref{it:2} is very mild
and easily satisfied in most situations. For instance,
when $f$ is the probability density of a random variable $X$,
we can write, for every $\epsilon > 0$,
\begin{equation} \label{eq:RVX}
	\int_\R |x|^\epsilon f(x) \, \dd x \,=\,
	\e(|X|^\epsilon) \,=\, \int_0^\infty \p(|X|^\epsilon > s) \, \dd s
	\,=\, \epsilon \, \int_0^\infty \frac{\p(|X| > t)}{t^{1-\epsilon}} \, \dd t \,.
\end{equation}
It follows, in particular, that \emph{assumption~\eqref{it:2} always holds
for the density $f$ of a random variable $X$
in the domain of attraction of a stable law, of any index $\alpha \in (0,2]$},
because it is well-known that in this case $\e(|X|^\epsilon) < \infty$
for every $0 < \epsilon < \alpha$.
\end{remark}

\begin{remark}\rm
We don't know whether assumption \eqref{it:2}
can be completely eliminated from Theorem~\ref{th:main}.
In other terms, we are not aware of examples of functions $f \in L^1$ satisfying assumption
\eqref{it:1} --- and necessarily \emph{not} satisfying assumption
\eqref{it:2} --- such that no convolution $f_k$ is d.R.i., for any $k\in\N$.
We point out that, if they exist, 
such counterexamples can be found in the class of bounded and continuous functions
that vanish at infinity, because assumption \eqref{it:1} entails that $f_{k_0 + 1}$
has these properties, by Lemma~\ref{th:tohol} below.
\end{remark}

%In fact, in this case $\p(|X| > t) \le L(t)\, t^{-\alpha}$ for every $t \ge 0$, where
%$L(\cdot)$ is a suitable slowly varying function; since slowly varying functions are
%asymptotically bounded by any polynomial, for large $t$ we have
%$\p(|X| > t) \le t^{-\alpha/2}$. Looking back at \eqref{eq:RVX},
%it is clear that for $\epsilon < \frac{\alpha}{2}$ the integral converges.

%An easy sufficient condition is clearly that
%\begin{equation*}
%	\exists \epsilon, C, x_0 \in (0,\infty): \qquad
%	|f(x)| \,\le\, \frac{C}{|x|^{1+\epsilon}} \,, \qquad
%	\forall x \in \R: \ |x| > x_0 \,.
%\end{equation*}

%The main result of this note is that a small moment condition
%is sufficient.
%
%the following assumptions:
%\begin{itemize}
%\item $f(\cdot)$ is essentially bounded, i.e. $\| f \|_\infty :=
%\mathrm{ess\, sup}_{x\in\R} f(x) < \infty$;
%
%\item $f(\cdot)$ has a finite moment, i.e., there exists $\alpha > 0$ such that
%\begin{equation} \label{eq:hyp0}
%	C_1 := \int_\R |x|^\alpha f(x) \, \dd x < \infty \,.
%\end{equation}
%\end{itemize}
%
%$f(x) \ge 0$ for a.e. $x \in \R$ and $f \not\equiv 0$ a.e. (that is
%$\int_\R f(x) \, \dd x \ne 0$);
%
%We are going to show that
%for every $\delta > 0$ there exists $k \in \N$ large enough such that
%\begin{equation} \label{eq:aimm}
%	\sum_{m \in \Z} \delta \cdot \sup_{z \in [m\delta, (m+1)\delta)}
%	\big( (1+|z|)^{\alpha} \,  f_k(z) \big) \,<\, \infty \,.
%\end{equation}
%This implies in particular that $f_k(\cdot)$ is directly Riemann integrable.

\section{Applications to renewal theory}
\label{sec:apprenth}

If $\{X_n\}_{n\in\N}$ are independent, identically distributed 
\emph{non-negative} random variables, the associated random walk
started at zero, that is $S_0 := 0$ and $S_n := X_1 + \ldots + X_n$
for every $n\in\N$, is called (undelayed) \emph{renewal process}.
The corresponding renewal measure $U(\cdot)$ is the $\sigma$-finite Borel measure
on $[0,\infty)$ defined by
\begin{equation} \label{eq:U}
	U(A) \,:=\, \e\big(\#\{n \in \N_0: \ S_n \in A\}\big) \,=\,
	\sum_{n\in\N_0} \p(S_n \in A) \,, \quad \text{for Borel }
	A \subseteq [0,\infty) \,.
\end{equation}

When $\mu := \e(X_1) \in (0,\infty)$ and the law of $X_1$
is non-lattice, Blackwell's renewal theorem states that,
for every fixed $\delta > 0$,
\begin{equation} \label{eq:Blackwell}
	\lim_{x\to+\infty} U\big([x,x+\delta)\big) \,=\, \frac{\delta}{\mu} \,.
\end{equation}
This means that, roughly speaking, the measure $U(\cdot)$ is
asymptotically close to $\frac{1}{\mu}$ times the Lebesgue measure. It is therefore
natural to conjecture that, for suitable $g: [0,\infty) \to \R$,
\begin{equation} \label{eq:Key}
	\lim_{x\to+\infty} \int_{[0,\infty)} g(x-z) \, U(\dd z) \,=\, \frac{1}{\mu}
	\Bigg( \int_{[0,\infty)} g(z) \, \dd z \Bigg) \,.
\end{equation}
This relation indeed holds \emph{whenever $g$ is d.R.i.}
(but can fail for general $g \in L^1$) and is known as the \emph{key renewal theorem}.
This is how d.R.i. functions appear in renewal theory.

Note that taking
$g = \ind_{(0,\delta]}$ one recovers Blackwell's renewal theorem \eqref{eq:Blackwell}.
We point out that relations \eqref{eq:Blackwell}, \eqref{eq:Key} hold
also when $\mu = +\infty$, the right hand side being interpreted as zero,
but of course they give much less information (we come back on this point below).
We refer to \cite[\S V]{cf:Asm} for more details on renewal theory.

%Then for every $g \in L^1$ we can write
%\begin{equation*}
%	\int_{[0,\infty)} g(x-z) \, U(\dd z) \,=\, g(x) \,+\, 
%	\sum_{n=1}^\infty (f_n * g)(x) \,.
%\end{equation*}
%
%Since $\mu = \int_0^\infty x \, f(x) \, \dd x < \infty$, 
%assumption~\eqref{it:2} of Theorem~\ref{th:main} is satisfied.
%If also assumption~\eqref{it:1} is satisfied, then
%there exists $k < \infty$ such that $f_{k}$ is
%d.R.i..
%\begin{equation} \label{eq:u}
%	u(x) \,=\, \sum_{n\in\N} f_n(x) \,.
%\end{equation}

\subsection{On the renewal density theorem}

Consider now the case when the law of $X_1$ is absolutely continuous,
with density $f$, and always assume that $\mu = \e(X_1) \in (0,\infty)$. Then
\begin{equation} \label{eq:Uabscont}
	U(\dd x) \,=\, \delta_0(\dd x) \,+\, \Bigg( \sum_{n=1}^\infty f_n(x) \Bigg) \dd x
	\,=:\, \delta_0(\dd x) \,+\, u(x) \, \dd x \,,
\end{equation}
where we recall that $f_n = f^{*n}$ denotes the $n$-fold convolution of $f$
with itself, cf. \eqref{eq:conv}. 
Therefore, excluding the Dirac mass at zero due to $S_0$,
the renewal measure $U(\cdot)$ is absolutely continuous, with density $u(x)$.
It is tempting to deduce from \eqref{eq:Blackwell}
the corresponding relation for the density, namely
\begin{equation}\label{eq:Density}
	\lim_{x\to+\infty} u(x) \,=\, \frac{1}{\mu} \,,
\end{equation}
sometimes called \emph{renewal density theorem}.
However, additional conditions on $f$ are needed for \eqref{eq:Density}
to hold: for instance,
\emph{it is necessary that $\lim_{x\to+\infty} f(x) = 0$}, as proved
by Smith~\cite[\S 4]{cf:SmithNS} (generalizing an earlier result by
Feller~\cite{cf:FelOld}). This is rather intuitive, because
any fixed term in the sum \eqref{eq:U} gives no
contribution to the asymptotic behavior of $U([x,x+\delta))$ ---
%as $x\to\infty$, 
since $\lim_{x\to+\infty}\p(S_n \in [x,x+\delta)) = 0$ for every fixed $n\in\N$ ---
while this is not the case for $u(x)$ 
if the density $f$ does not vanish at infinity, cf.~\eqref{eq:Uabscont}.
%Therefore, \emph{until the rest of this
%subsection, we make the assumption that $\lim_{x\to+\infty} f(x) = 0$},
%which easily implies $\lim_{x\to+\infty} f_n(x) = 0$, for every fixed $n\in\N$.

Sharp necessary and sufficient conditions on $f$ for the validity
of the renewal density theorem \eqref{eq:Density} are known
\cite{cf:SmithNS}, but they are quite involved and implicit. A natural
sufficient condition \cite{cf:SmithOld1,cf:SmithOld2}
is simply that $f \in L^p$ for some $p \in (1,\infty]$
(in addition to $\mu = \int_\R x \, f(x) \, \dd x
\in (0,\infty)$ and $\lim_{x\to+\infty} f(x) = 0$).
It is worth noting that the sufficiency of these conditions is an immediate
corollary of our Theorem~\ref{th:main}: in fact, if $\mu < \infty$, 
assumption \eqref{it:2} is satisfied with $\epsilon = 1$, and if
$f \in L^p$ with $p > 1$,
assumption \eqref{it:1} is also satisfied, by Lemma~\ref{th:suf}; 
it follows that $f_k$ is d.R.i. for some $k\in\N$, and by
the key renewal theorem \eqref{eq:Key} we have
\begin{equation} \label{eq:appldRi}
	\lim_{x\to+\infty} \int_\R f_k(x-z) \, U(\dd z) \,=\, 
	\lim_{x\to+\infty} \Bigg( \sum_{n = k}^\infty f_n(x) \Bigg) \,=\, 
	\frac{1}{\mu}
	\int_\R f_k(z) \, \dd z \,=\, \frac{1}{\mu} \,,
\end{equation}
where in the second equality we have used \eqref{eq:Uabscont} and the fact
that $f_{i} * f_{j} = f_{i+j}$. We can rewrite this relation as
\begin{equation} \label{eq:u-f}
	\lim_{x\to+\infty} \Bigg( u(x) \,-\, \sum_{n = 1}^{k-1} f_n(x) \Bigg) \,=\, 
	\frac{1}{\mu} \,.
\end{equation}
It follows easily by \eqref{eq:conv} that
\begin{equation}\label{eq:usefu}
	f_{n+1}(x) \,\le\, \sup_{y \ge x/2} \big( f(y) + f_n(y) \big) \,, \qquad
	\forall x \in [0,\infty), \ \forall n \in \N \,.
\end{equation}
If $\lim_{x\to+\infty} f(x) = 0$,
relation \eqref{eq:usefu} shows by induction
that also $\lim_{x\to+\infty} f_n(x) = 0$, for every fixed $n\in\N$,
hence relation \eqref{eq:Density} follows from \eqref{eq:u-f}.

\begin{remark}\rm
The idea of deriving the renewal density theorem
\eqref{eq:Density} from the key renewal theorem \eqref{eq:Key} is a
classical one, dating back at least to Feller, who proved the validity
of \eqref{eq:u-f} with $k = 2$ when $f \in L^\infty$, 
showing that $f_2$ is d.R.i., cf. Theorem~2a in~\cite[\S XV.3]{cf:Fel2}. 
%The key ingredient in Feller's (simple) proof is that when the density $f$ is bounded
%\emph{and has finite mean}, then $f_2$ is d.R.i. and one can apply the key renewal theorem,
%precisely as we did in \eqref{eq:appldRi}. 
Feller's proof is based on the simple observation that, by \eqref{eq:conv}
and a symmetry argument,
\begin{equation} \label{eq:feller}
	f_2(x) \,=\, 2 \int_{z > x/2} f(z) \, f(x-z) \, \dd z \,\le\,
	2 \, \|f\|_\infty \, \int_{z > x/2} f(z) \,=\, 
	2 \, \|f\|_\infty \, \p\big( X_1 > \tfrac{x}{2} \big) \,.
\end{equation} 
Since by assumption
\begin{equation*}
	\int_0^\infty \p\big( X_1 > \tfrac{x}{2} \big) \, \dd x \,=\, 2 \e(X_1)
	\,\in\, (0,\infty) \,,
\end{equation*}
the function $f_2$ is dominated
by a non-increasing, integrable function, hence it is d.R.i..

We point out that the generalization that we sketched above,
in which the assumption $f \in L^\infty$ is relaxed to $f \in L^p$ for some $p > 1$, 
is a rather elementary upgrade: since
a convolution $f_k$ of $f$ is bounded, thanks to Lemma~\ref{th:suf},
applying relation \eqref{eq:feller} to $f_k$ yields immediately that $f_{2k}$ is d.R.i.,
allowing to deduce \eqref{eq:u-f} (with $k$ replaced by $2k$) from
the key renewal theorem \eqref{eq:Key}. Let us stress, however, that to deduce
direct Riemann integrability from \eqref{eq:feller}, the
\emph{finiteness of the mean $\e(X_1)$} is essential.
\end{remark}

\subsection{The heavy-tailed case}

The novelty of Theorem~\ref{th:main} is that assumption \eqref{it:2} only requires
the finiteness of an arbitrarily small moment,
allowing in particular to deal with cases when the mean is infinite.
This is especially interesting from the viewpoint of heavy-tailed renewal theory.
More precisely, keeping the notation of the beginning of this section,
assume that the law of $X_1$ is non-lattice and satisfies the following relation,
for some $\alpha \in (0,1]$:
\begin{equation} \label{eq:tailX1}
	\p(X_1 > x) \,\sim\, \frac{L(x)}{x^\alpha} \qquad
	\text{as } x \to +\infty \,,
\end{equation}
where $L(\cdot)$ is a slowly varying function \cite{cf:BinGolTeu}.
For $\alpha < 1$, relation \eqref{eq:tailX1} is the same as requiring that
$X_1$ is in the domain of attraction of the (unique up to multiples)
positive stable law of index $\alpha$ (cf. Theorem~8.3.1
in~\cite{cf:BinGolTeu}), while for $\alpha = 1$
relation \eqref{eq:tailX1} implies that $X_1$ is relatively stable
(cf. Theorem~8.8.1, Corollary~8.1.7 and the following lines in~\cite{cf:BinGolTeu}).

When $\mu = \e(X_1) = +\infty$ (in particular, for every $\alpha < 1$),
Blackwell's renewal theorem \eqref{eq:Blackwell}
only says that $\lim_{x\to+\infty} U([x,x+\delta)) = 0$. Sharpenings
of this relation have been proved by Erickson~\cite[Theorems 1--4]{cf:Eri},
extending the corresponding results for lattice distributions 
obtained by Garsia and Lamperti~\cite{cf:GarLam}.
Introducing the truncated mean function
\begin{equation} \label{eq:intmean}
	m(x) \,:=\, \int_0^x \p(X_1 > y) \, \dd y \,=\,
	x \, \p(X_1 > x) \,+\, \e(X_1 \ind_{\{X_1 < x\}}) \,, \qquad \forall x \ge 0 \,,
\end{equation}
it follows from \eqref{eq:tailX1} that $m(x) \sim \frac{1}{1-\alpha} L(x) x^{1-\alpha}$
as $x\to+\infty$. The generalized version of Blackwell's renewal theorem then
reads as follows: for every fixed $\delta > 0$
\begin{equation} \label{eq:Blackwellalpha}
	\liminf_{x \to +\infty} \, m(x) \, U\big([x,x+\delta) \big) \,=\,
	\frac{1}{\Gamma(\alpha) \Gamma(2-\alpha)} \, \delta \,,
\end{equation}
and when $\alpha \in (\frac{1}{2}, 1]$ the $\liminf$ in this relation
can be upgraded to a true limit.
A generalized version of the key renewal theorem is also available: for every
d.R.i. function $g: [0,\infty) \to \R$
\begin{equation} \label{eq:Keyalpha}
	\liminf_{x \to +\infty} \, m(x) \, \int_{[0,\infty)} g(x-z) \, U(\dd z) \,=\,
	\frac{1}{\Gamma(\alpha) \Gamma(2-\alpha)} \,
	\Bigg( \int_{[0,\infty)} g(y) \, \dd y \Bigg) \,.
\end{equation}
Furthermore, when $\alpha \in (\frac{1}{2}, 1]$ and
$g(x) = O(1/x)$ as $x \to +\infty$, also in this relation the $\liminf$
can be upgraded to a true limit. The reason for the presence
of $\liminf$ instead of $\lim$ is discussed in Remark~\ref{rem:liminflim} below.
Apart from that, note that relations \eqref{eq:Blackwellalpha},
\eqref{eq:Keyalpha} match perfectly with \eqref{eq:Blackwell}, \eqref{eq:Key},
because $\mu = \e(X_1) = m(\infty)$.

Assume now that $X_1$ is absolutely continuous, with a density $f$.
The density $u$ of the renewal measure $U$ is always defined by \eqref{eq:Uabscont},
and it is natural to ask whether the density version of \eqref{eq:Blackwellalpha}
holds true. As a corollary of Theorem~\ref{th:main}, we obtain the following
result, which seems to be new.

\begin{proposition} \label{th:cor}
Let $X_1$ be a non-negative random variable satisfying \eqref{eq:tailX1},
for $\alpha \in (0,1]$ and $L(\cdot)$ slowly varying.
Assume that the law of $X_1$ is absolutely continuous, with a density
$f$ such that $f_k \in L^\infty$ for some $k\in\N$ (cf. Lemmas~\ref{th:necsuf}
and~\ref{th:suf}). Then, recalling the definitions \eqref{eq:Uabscont} of $u(\cdot)$
and \eqref{eq:intmean} of $m(\cdot)$, there exists $\overline{k} \in \N$ such that
\begin{equation} \label{eq:new1}
	\liminf_{x \to +\infty} \, m(x) \, \Bigg( u(x) \,-\, 
	\sum_{n = 1}^{\overline{k}-1} f_n(x) \Bigg) \,=\,
	\frac{1}{\Gamma(\alpha) \Gamma(2-\alpha)} \,.
\end{equation}
If furthermore $\lim_{x\to+\infty} f(x) = 0$, then 
\begin{equation} \label{eq:new2}
	\liminf_{x \to +\infty} \, m(x) \, u(x) \,=\,
	\frac{1}{\Gamma(\alpha) \Gamma(2-\alpha)} \,.
\end{equation}
Finally, if $\alpha \in (\frac{1}{2}, 1]$ and $f_k(x) = O(1/x)$
as $x\to+\infty$, for some $k\in\N$,
relations \eqref{eq:new1} and \eqref{eq:new2} hold with $\lim$ instead of $\liminf$.
\end{proposition}

\begin{proof}
The density $f$ satisfies the assumptions of Theorem~\ref{th:main}, hence
there exists $k_1 \in \N$ such that
$f_{k}$ is d.R.i. for every $k\ge k_1$. Choosing $\overline{k} = k_1$
and applying \eqref{eq:Keyalpha}
with $g = f_{\overline{k}}$, we obtain immediately \eqref{eq:new1}.
We have already remarked that, by \eqref{eq:usefu}, 
if $\lim_{x\to+\infty} f(x) = 0$,
then also $\lim_{x\to+\infty} f_n(x) = 0$, for every fixed $n\in\N$,
hence relation \eqref{eq:new2} follows from \eqref{eq:new1}.
Finally, if $\alpha \in (\frac{1}{2}, 1]$
and $f_k(x) = O(1/x)$ for some $k \in \N$, then also
$f_{nk}(x) = O(1/x)$ for every $n\in\N$, as it follows 
by induction from \eqref{eq:usefu} applied to $f_k$. If $n$ is large
enough so that $\overline{k} := nk \ge k_1$, the function $f_{\overline{k}}(x)$
is both d.R.i. and $O(1/x)$, hence relation \eqref{eq:Keyalpha} holds with
$\lim$ instead of $\liminf$ and consequently the same is true for
\eqref{eq:new1}, \eqref{eq:new2}.
\end{proof}

\begin{remark}\rm
It is not clear whether the additional condition $g(x) = O(1/x)$, for the validity
of the generalized key renewal
theorem \eqref{eq:Keyalpha} with $\lim$ instead of $\liminf$
when $\alpha \in (\frac{1}{2}, 1]$, 
formulated in \cite[Theorem~3]{cf:Eri}, is substantial or just technical.
In any case, if that condition can be removed or relaxed, the same
applies immediately to Proposition~\ref{th:cor}.
\end{remark}

\begin{remark}\rm\label{rem:liminflim}
The presence of $\liminf$ instead of $\lim$, in relations \eqref{eq:Blackwellalpha}
and \eqref{eq:Keyalpha} when $\alpha \in (0,\frac{1}{2}]$ but not when
$\alpha \in (\frac{1}{2}, 1]$, may appear strange, but can be explained
as follows. By \eqref{eq:tailX1},
$\p(X_1 \in [x,x+\delta)) \le \p(X_1 \ge x) \sim L(x)/x^{\alpha}$ as $x\to\infty$,
and a similar estimate (up to a constant)
holds for $\p(S_n \in [x,x+\delta))$, for every $n\in\N$.
Since $m(x) \sim \frac{1}{1-\alpha} L(x) x^{1-\alpha}$
as $x\to+\infty$, when $\alpha < \frac{1}{2}$ we have
$\p(S_n \in [x,x+\delta)) \ll 1 / m(x)$.
Recalling \eqref{eq:Blackwellalpha}, this means that any
term $\p(S_n \in [x,x+\delta))$,
for fixed $n\in\N$, gives a negligible contribution to the asymptotic
behavior of $U([x,x+\delta))$. This is no longer true when $\alpha > \frac{1}{2}$,
as one can build examples of laws of $X_1$ satisfying \eqref{eq:tailX1}
but such that $\p(X_1 \in [x,x+\delta))$ is
anomalously close to $\p(X_1 \ge x) \sim L(x)/x^{\alpha}$, for
(rare but) arbitrarily large values of $x$; in this way,
the contribution of the single term $\p(X_1 \in [x,x+\delta))$ can be made much larger
than the ``typical'' behavior of $U([x,x+\delta))$,
that is $1/m(x)$, by \eqref{eq:Blackwellalpha}.
For more details, we refer to \cite{cf:GarLam}, \cite{cf:Wil}.

In view of these considerations, it is natural to conjecture that,
under some additional regularity assumptions on the
distribution of $X_1$,
it should be possible to upgrade
the $\liminf$ in the generalized Blackwell's renewal theorem \eqref{eq:Blackwellalpha},
or in its density form \eqref{eq:new2},
to a true $\lim$ also for $\alpha \in (0,\frac{1}{2}]$. 
This is indeed the case, as shown recently
by Topchii in~\cite[Theorem~8.3]{cf:Top}
(generalizing the analogous results for lattice distributions
by Doney~\cite[Theorem~B]{cf:Don97}). More precisely, assume that $X_1$
satisfies \eqref{eq:tailX1}; that it has an absolutely continuous
law with density $f$, such that $f_k \in L^\infty$ for some $k\in\N$
(that is, assumption~\eqref{it:1} in Theorem~\ref{th:main}); and furthermore
that there exist positive constants $C, x_0$ such that
$f(x) \le C L(x) / x^{1+\alpha}$ for all $x \ge x_0$, where $L(\cdot)$ is the same
slowly varying function appearing in \eqref{eq:tailX1};
then the generalized renewal density theorem \eqref{eq:new2},
and consequently also the generalized Blackwell's theorem \eqref{eq:Blackwellalpha},
holds with $\lim$ instead of $\liminf$.
\end{remark}

\subsection{Application to local limit theorems}
\label{sec:appllt}

Beyond renewal theory, Theorem~\ref{th:main} can be used
to derive local limit theorems \emph{for the density} of a random walk,
even under conditioning,
when the corresponding local limit theorems \emph{\`a la Stone} are available.

For instance, let $\{X_n\}_{n\in\N}$ be
independent, identically distributed real random variables,
in the domain of attraction of a stable law,
and denote by $(S = \{S_n\}_{n\in\N}, \p_x)$ the associated
random walk started at $x\in\R$, that is $\p_x(S_0 = x) = 1$ and
$S_n := S_{n-1} + X_n$ for all $n\in\N$. Let us
also set $C_n := [0,\infty)^n \subseteq \R^n$.
When the law of $X_1$ is non-lattice,
%(that is, it is not supported by $\delta\Z$, for any $\delta > 0$), 
local limit theorems
in the Stone form --- that is, for the probabilities of small intervals ---
for the law of $S_n$ on the event $\{(S_1, \ldots, S_n) \in C_n\}$
are available, cf. \cite{cf:VatWac,cf:Don11}. For example, there exist
diverging sequences $(a_n)_{n\in\N}$, $(b_n)_{n\in\N}$
and real functions $\phi, \psi$ such that, for any \emph{fixed} $\delta > 0$,
as $n\to\infty$
\begin{equation} \label{eq:Stone}
	\p_x \big(S_n \in [y,y+\delta) \,,\, (S_1, \ldots, S_n) \in C_n \big) \,=\,
	\delta \, \frac{1}{b_n} \, \psi(x) \,
	\phi\bigg(\frac{y}{a_n}\bigg) \, \big( 1 + o(1) \big) \,,
\end{equation}
uniformly when $x/a_n \to 0$ and $y/a_n$ is bounded away both from $0$ and $\infty$.

Assume now that the law of $X_1$ is absolutely continuous, with density $f$,
and denote by $f_n^+$ the density of $S_n$ on the event $\{(S_1, \ldots, S_n) \in C_n\}$, i.e.
\begin{equation*}
	f_n^+(x,y) \,:=\, 
	\frac{\p_x \big(S_n \in \dd y \,,\, (S_1, \ldots, S_n) \in C_n \big)}{\dd y} \,.
\end{equation*}
It is then very natural to conjecture that the density version of \eqref{eq:Stone} holds, namely
\begin{equation} \label{eq:StoneDen}
	f_n^+(x,y) \,=\, \frac{1}{b_n} \, \psi(x) \,
	\phi\bigg(\frac{y}{a_n}\bigg) \, \big( 1 + o(1) \big) \,,
\end{equation}
but some care is needed in order to interchange the limits $\delta \downarrow 0$
and $n\to\infty$. In fact, \eqref{eq:StoneDen}
is not true in general, as the density of $S_n$,
and hence $f_n^+(x,y)$, might be unbounded for every $n\in\N$. However,
this turns out to be the only possible pathology.

If we assume that the density of $S_n$ is bounded
for some $n\in\N$, then Theorem~\ref{th:main} may be applied
(note that assumption~\eqref{it:2} is automatically satisfied, since
we assume that $X_1$ is in the domain of attraction of a stable law,
as we already remarked).
It follows that the density $f_k$ of $S_k$, 
and hence $z \mapsto f_k^+(z,y)$, is d.R.i. for some $k\in\N$,
and this allows to rigorously derive \eqref{eq:StoneDen} from \eqref{eq:Stone}.
We refer to \cite[\S 5]{cf:CarCha2} for the technical details,
but let us sketch the main idea, which is quite simple. For fixed $k \in \N$
and $n \ge k$ we can write
\begin{equation} \label{eq:decompint}
	f_n^+(x,y) \,=\, \int_{[0,\infty)} \dd z \, f_{n-k}^+(x,z) \, f_k^+(z,y) \,.
\end{equation}
Since $x \mapsto f_k^+(x,y)$ is d.R.i., we can effectively approximate it
with a step function, piecewise constant over disjoint intervals.
The integral in the right hand side of \eqref{eq:decompint} then becomes a sum
of terms, each of which is like the left hand side of \eqref{eq:Stone}, with
$n-k$ instead of $n$. Since $k$ is fixed, as $n\to\infty$ relation \eqref{eq:Stone}
holds and \eqref{eq:StoneDen} can be recovered from \eqref{eq:decompint}.

This approximation method is quite general, and may in principle be applied to
other contexts (e.g., for other choices of the conditioning subsets $C_n \subseteq \R^n$).
The message is that, whenever a local limit theorem in the Stone form is available,
Theorem~\ref{th:main} provides a helpful tool in deriving the corresponding
local limit theorem for the density.

\section{Proof of Theorem~\ref{th:main}}
\label{sec:proof}

Let us first discuss the strategy of the proof.
The starting point is Feller's observation \eqref{eq:feller}, 
which shows that $f_2$ can be bounded from above by a non-increasing function, namely,
the integrated tail of $f$. 
When $f$ has finite mean (that is, assumption \eqref{it:2} holds with
$\epsilon =1$), it follows that $f_2$ is d.R.i., but when the mean is infinite
this bound is not enough.

The natural idea is then to bootstrap the estimate \eqref{eq:feller}, 
deducing an estimate on $f_4$ from the bound on $f_2$, and so on,
hoping that convolutions are regularizing enough so that for some $n\in\N$
the bound obtained for $f_{2^n}$ yields the direct Riemann integrability.
This turns out to be the case, though in a highly non straightforward way.

For convenience, we organize the proof in four steps.

\subsection{Some preliminary results}

Let us give a name to the (translated) upper and lower Riemann sums of
a function $g : \R \to \R$: for $\delta \in (0,\infty)$ and $x\in\R$ we set
\begin{equation} \label{eq:SsR}
	S^g_\delta(x) \,:=\,
	\sum_{m\in\Z} \delta \, \bigg(  \sup_{z \in [m\delta, (m+1)\delta)} g(z-x) \bigg) \,,
	\quad
	s^g_\delta(x) \,:=\,
	\sum_{m\in\Z} \delta \, \bigg(  \inf_{z \in [m\delta, (m+1)\delta)} g(z-x) \bigg) \,.
\end{equation}
Note that both $S^g_\delta(x)$ and $s^g_\delta(x)$ are $\delta$-periodic functions of $x$.
Moreover, when $g$ is non-negative, the following inequality holds,
as we prove in \S\ref{sec:proofbasto}:
\begin{equation} \label{eq:basto}
	S^g_{\delta}(x) \,\le\, \bigg( 1 + 2 \frac{\delta}{\delta'} \bigg) S^g_{\delta'}(x') \,,
%	\quad
%	s^g_{\delta}(x) \,\ge\, \bigg( 1 - 2 \frac{\delta}{\delta'} \bigg) s^g_{\delta'}(x') \,,
	\qquad \forall x,x' \in \R\,, \ \forall \delta, \delta' > 0 \,.
\end{equation}
This shows in particular that the finiteness of $S^g_{\delta}(x)$ does not depend on $\delta, x$.

Recall that, by equation \eqref{eq:dRi}, a non-negative function
$g$ is d.R.i. if and only if $S^g_\delta(0)$ and $s^g_\delta(0)$
converge to the same finite limit as $\delta \downarrow 0$.
It is an easy exercise to prove
the following lemma, which provides a useful reformulation of the d.R.i. condition.

%We are going to show that
%the same holds for $S^g_\delta(x)$ and $s^g_\delta(x)$, \emph{uniformly} over $x\in\R$.

\begin{lemma}\label{th:reformulation}
A function $g: \R \to \R$ is d.R.i. if and only if
it is Riemann integrable
%(in the ordinary sense) 
on every compact interval $[a,b] \subseteq \R$,
and if in addition the upper Riemann sum
$S_\delta^{|g|}(x)$ of $|g|$ is finite for some 
(hence all) $x \in \R$ and $\delta > 0$.
In particular, every non-negative, continuous function with a finite
upper Riemann sum is d.R.i..
\end{lemma}

%\begin{remark}\label{rem:usefu}\rm
%A useful equivalent reformulation is as follows:
%a non-negative function $g$ is d.R.i. if it is Riemann integrable
%%(in the ordinary sense) 
%on every compact interval $[a,b] \subseteq \R$,
%and if in addition its $\delta$-upper Riemann sum
%--- that is, the expression in the left hand side of \eqref{eq:dRi}, before taking the limit ---
%is finite for some (hence any) $\delta > 0$.
%In particular, \emph{every non-negative, continuous function, with a finite
%upper Riemann sum, is d.R.i.}.
%\end{remark}

%\begin{lemma}\label{th:dRiconv}
%Let $g,h: \R \to \R$ be such that
%$g$ is d.R.i. and $h \in L^1$. Then $g * h$ is d.R.i..
%\end{lemma}

The following standard result will also be useful.

\begin{lemma}\label{th:tohol}
If $g \in L^1$ and $h \in L^1 \cap L^\infty$,
then $g * f$ is bounded, continuous and vanishes at infinity
(that is $\lim_{|x| \to +\infty} (g * f)(x) = 0$).
\end{lemma}

\begin{proof}
If $g(x) = \sum_{i=1}^n c_i \ind_{(a_i, b_i]}(x)$ is a step function, the theorem
holds by direct verification, since 
$(g*h)(x) = \sum_{i=1}^n c_i \int_{[x - b_i, x - a_i)} h(z) \, \dd z$ and $h \in L^1$.
For every $g \in L^1$, take a sequence of step functions $g_n$
such that $\|g-g_n\|_1 \to 0$. Since
$|(g*h)(x) - (g_n*h)(x)| \le \|h\|_\infty \|g-g_n\|_1$ for every $x \in \R$,
$g*h$ is the uniform limit of $g_n * h$ and the conclusion follows.
\end{proof}

\subsection{Setup}
\label{sec:prep}

If $f \in L^1(\R,Leb)$ satisfies
assumption~\eqref{it:1} of Theorem~\ref{th:main}, it follows by Lemma~\ref{th:tohol} that
that $f_k = f^{*k}$ is bounded and continuous
(and vanishes at infinity) for every $k \ge k_0 + 1$. 
By Lemma~\ref{th:reformulation},
to prove Theorem~\ref{th:main} it suffices to show that an upper Riemann sum
of $x \mapsto (1+|x|^\epsilon) |f_k(x)|$, say with mesh $1$, is finite, that is
\begin{equation} \label{eq:aimm}
	\sum_{m \in \Z} \bigg( \sup_{z \in [m, (m+1))}
	(1+|z|^{\epsilon}) \,  |f_k(z)| \bigg) \,<\, \infty \,,
\end{equation}
for all $k$ large enough. Actually, if this relation holds for
$k = \overline {k}$, one easily
shows that it holds for every $k \ge \overline{k}$, cf.~\S\ref{sec:bootaimm}.
Therefore it suffices to prove \eqref{eq:aimm} for some $k\in\N$.

%Note that this relations implies that the function $z \mapsto (1+|z|)^{\epsilon} \,  |f_k(z)|$
%is bounded.

\smallskip

Since $|f_k| \le |f|_k$, that is $|f^{*k}| \le |f|^{*k}$, 
without loss of generality we may assume
that the function $f$ is non-negative (it suffices to replace $f$ by $|f|$).
Moreover, excluding the trivial case when $f = 0$ almost everywhere,
in which there is nothing to prove, we may also
impose the normalization $\int_\R f(x) \, \dd x = 1$
(it suffices to multiply $f$ by a constant). In this way,
$f$ may be viewed as a probability density.
As a consequence,
also $f_k$ is a probability density:
$f_k \ge 0$ and $\int_\R f_k(x) \, \dd x  = 1$, for all $k\in\N$.
Let us set for convenience (recall assumption~\eqref{it:2})
\begin{equation*}
	C \,:=\, \int_\R |x|^\epsilon\, f(x) \, \dd x \,<\, \infty \,.
\end{equation*}
Since $|x_1 + \ldots + x_k|^\epsilon \le (k \, \max_{1 \le i \le k}|x_i|)^\epsilon
\le k^\epsilon (|x_1|^\epsilon +\ldots + |x_k|^\epsilon)$, for every $k\in\N$
\begin{equation} \label{eq:momentk}
\begin{split}
	\int_\R |x|^\epsilon \, f_k(x) \, \dd x
	& \,=\, \int_{\R^k} |x_1 + \ldots + x_k|^\epsilon \, 
	f(x_1) \cdots f(x_k) \, \dd x_1 \cdots \dd x_k 
	\,\le\, k^{\epsilon} (k C) \,<\, \infty \,.
\end{split}
\end{equation}
It follows in particular that $f_{k}$ satisfies the hypothesis of Theorem~\ref{th:main},
for every $k\in\N$.
Therefore, we may assume that $f \in L^\infty$ (it suffices to replace $f$ by $f_{k_0}$).

\smallskip

Summarizing, without any loss of generality, henceforth \emph{we assume that $f$ is a
bounded probability density}, that is $f: \R \to \R$ is a measurable function such that
\begin{equation}\label{eq:newass}
	f(x) \ge 0 \quad \forall x \in \R \,, \qquad
	\int_\R f(x) \, \dd x = 1 \,, \qquad
	\sup_{x\in\R} f(x) \,<\, \infty \,,
\end{equation}
and our goal is to show that \eqref{eq:aimm} holds true for some $k \in \N$.

\subsection{A sequence of upper bounds}

By Markov's inequality, for all $t\ge 0$ we can write
\begin{equation} \label{eq:ass'}
	g_1(t) \,:=\, \int_{|x| \ge t} f(x) \, \dd x 
	\,\le\, \min\bigg\{ 1, \,
	\frac{1}{t^{\epsilon}} \int_{|x| \ge t} |x|^{\epsilon} f(x) \, \dd x \bigg\}
	\,\le\, \min\bigg\{1, \frac{C}{t^\epsilon}\bigg\} \,.
\end{equation}
Analogously, for every $k\in\N$, since
$\ind_{\{|x_1 + \ldots + x_k| \ge t\}} \le \sum_{i=1}^k \ind_{\{|x_i| \ge t/k\}}$,
for all $t\ge 0$ we obtain
\begin{equation} \label{eq:assn'}
\begin{split}
	g_k(t) & \,:=\, \int_{|x| \ge t} f_k(x) \, \dd x \,=\, 
	\int_{\R^k} f(x_1) \cdots f(x_k) \, 	\ind_{\{|x_1 + \ldots + x_k| \ge t\}}
	\, \dd x_1 \cdots \dd x_k \\
	& \,\le\, \min\bigg\{1,  k \, g_1\bigg( \frac{t}{k} \bigg) \bigg\}
	\,\le\, \min\bigg\{1, \frac{k^{1+\epsilon} C}{t^\epsilon} \bigg\} \,.
\end{split}
\end{equation}
For consistency and for later convenience, for $t \le 0$ we set 
$g_k(t) := g_k(0) = 1$, for all $k \in \N$.
It follows from \eqref{eq:assn'} that for all $k\in\N$
\begin{equation} \label{eq:integrabl}
	\int_\R g_{k} \bigg( \frac{|w|}{3} - 1 \bigg)^{p} \, \dd w
	\,=\, 6 \, g_k(0) + 2 \cdot 3 \int_0^\infty g_k(t)^p \, \dd t
	\,<\, \infty \,, \qquad \forall p \,>\, \frac{1}{\epsilon} \,.
\end{equation}

For every $n\in\N$, we define a map
$\Phi_{n}: L^\infty \to L^\infty$ by
\begin{equation} \label{eq:defPhi}
	(\Phi_{n}(h))(x) \,:=\, 2 \int_{|z| > (|x| - 3)/2} \dd z \,
	f_{n}(z) \, h(x-z) \,, \qquad \forall x \in \R \,.
\end{equation}
The reason for such a definition is explained by the following crucial lemma,
which will enable us to bound
the left hand side of \eqref{eq:aimm} by the integral of a suitable function.

\begin{lemma} \label{th:basic}
Let $h \in L^\infty(\R, Leb)$, with $h \ge 0$, and $\ell \in\N$ be such that
\begin{equation} \label{eq:hyy}
	\sup_{x \in [a,a+1)} f_{\ell}(x) \le
	\int_{a-1}^{a+2} h(w) \, \dd w \,, \qquad
	\forall a \in \R \,.
\end{equation}
Then
\begin{equation} \label{eq:thh}
	\sup_{x \in [a,a+1)} f_{2\ell}(x) \le
	\int_{a-1}^{a+2} \Phi_{\ell}(h)(w) \, \dd w \,, \qquad
	\forall a \in \R \,.
\end{equation}
\end{lemma}

\begin{proof}
We start giving a couple of slight generalizations of \eqref{eq:hyy}.
First, by a straightforward translation argument,
\begin{equation} \label{eq:gen1}
	\sup_{x \in [a,a+1)} f_{\ell}(x - z) \,\le\,
	\int_{a-1}^{a+2} h(w - z) \, \dd w \,, \qquad
	\forall a,z \in \R\,.
\end{equation}
Next we claim that, for every $z \in \R$ and $t > 0$,
\begin{equation} \label{eq:gen2}
	\sup_{x \in [a,a+1)} \Big( f_{\ell}(x - z) \ind_{\{|x| < t\}} \Big)
	\le \int_{a-1}^{a+2} h(w - z) \ind_{\{|w| < t + 3\}}
	\, \dd w \,, \qquad
	\forall a,z \in \R \,, \ \forall t > 0\,.
\end{equation}
In fact, if $t < |a| - 1$ then the left hand side of \eqref{eq:gen2} is zero
and there is nothing to prove. On the other hand, if $t \ge |a| - 1$
the right hand sides of \eqref{eq:gen2} and \eqref{eq:gen1} coincide,
because $\ind_{\{|w| \le t + 3\}} = 1$ for all $w$ in the domain
of integration $(a-1, a+2)$, 
hence \eqref{eq:gen2} follows from \eqref{eq:gen1}
simply because $f_{\ell}(x - z) \ind_{\{|x| \le t\}} \le f_{\ell}(x - z)$.

Since $f_{2\ell} = f_\ell * f_\ell$, for all $x \ge 0$ we can write
\begin{equation*}
	f_{2\ell}(x) = \int_\R f_\ell(z) f_\ell(x-z) \, \dd z
	= 2 \int_{z > x/2} f_\ell(z) f_\ell(x-z) \, \dd z \,,
\end{equation*}
having exploited the symmetry $z \leftrightarrow x-z$.
For $x \le 0$ we can write an analogous formula, with $\{z > x/2\}$
replaced by $\{z < -x/2\}$. We can combine these relations in
the following single inequality:
\begin{equation} \label{eq:easyboth}
	f_{2\ell}(x) \le 2 \int_{|z| > |x|/2} f_\ell(z) f_\ell(x-z) \, \dd z \,,
	\qquad \forall x\in\R \,,
\end{equation}
hence
\begin{equation*}
	\sup_{x \in [a,a+1)} f_{2\ell}(x) 
	\le 2 \int_\R \dd z \, f_\ell(z) 
	\sup_{x \in [a,a+1)} \Big( f_\ell(x-z)
	\ind_{\{|x| < 2|z|\}} \Big) \,, \qquad
	\forall a,x \in \R, \ \forall t > 0 \,.
\end{equation*}
We now apply \eqref{eq:gen2}, getting
\begin{equation*}
\begin{split}
	\sup_{x \in [a,a+1)} f_{2\ell}(x) 
	& \le 2 \int_\R \dd z \, f_\ell(z) 
	\int_{a-1}^{a+2} \dd w \, h(w - z) \ind_{\{|w| < 2|z| + 3\}} \\
	& = \int_{a-1}^{a+2} \dd w \, \bigg( 2
	\int_\R \dd z \, f_\ell(z) \,  h(w - z) \ind_{\{|w| < 2|z| + 3\}} \bigg) \,,
\end{split}
\end{equation*}
which, recalling \eqref{eq:defPhi}, is exactly \eqref{eq:thh}.
\end{proof}

Applying iteratively Lemma~\ref{th:basic} we now obtain a sequence
of upper bounds for $f_{2^n}(\cdot)$. Let us start with $n=1$, i.e.,
with $f_{2^1}(\cdot) = f_2(\cdot)$, showing that \eqref{eq:hyy}
holds for a suitable choice of $h(\cdot)$.
%$= \sup\{t \ge 0: \ Leb(\{x \in \R:\, |h(x)|>t\}) > 0\}$.
Recalling \eqref{eq:easyboth} and \eqref{eq:ass'}, for all $x \in \R$ we can write
\begin{equation*}
	f_2(x) \,\le\, 2 \int_{|z| > |x|/2} \dd z \, f(z) \, f(x-z)
	\,\le\, 2 \, \|f\|_\infty \int_{|z| > |x|/2} \dd z \, f(z)
	\,=\, 2\, \|f\|_\infty\, g_1(|x|/2) \,.
\end{equation*}
The function $x \mapsto g_1(|x|/2)$ is non-decreasing for $x \le 0$ and non-increasing
for $x \ge 0$, hence for every $a \ge 1$ we have
\begin{equation*}
	\sup_{x \in [a, a + 1)} f_2(x) \le 2\, \|f\|_\infty\, g_1(a/2) 
	\le \int_{a-1}^{a} 2\, \|f\|_\infty\, g_1(|w|/2)  \, \dd w \,,
\end{equation*}
and analogously for $a \le -2$
\begin{equation*}
	\sup_{x \in [a, a + 1)} f_2(x) \le 2\, \|f\|_\infty\, g_1(|a+1|/2) 
	\le \int_{a+1}^{a+2} 2\, \|f\|_\infty\, g_1(|w|/2)  \, \dd w \,.
\end{equation*}
Altogether, for every $a \in \R \setminus [-2, 1]$
\begin{equation} \label{eq:quasi}
	\sup_{x \in [a,a+1)} f_2(x) \le
	\int_{a - 1}^{a + 2} 2\, \|f\|_\infty\, g_1(|w|/2)  \, \dd w \,.
\end{equation}

Let us show that an analogous estimate holds also for $a \in [-2, 1]$.
Note that the right hand side of \eqref{eq:quasi} is a continuous function of $a$.
Furthermore, it is strictly positive
for $a \in [-2, 1]$, because 
$g_1(|w|/2)$ is strictly positive in a neighborhood of $w=0$
--- it is continuous and it equals one in zero, cf. \eqref{eq:ass'}
--- and $0$ is in the domain of
integration $[a-1, a+2]$ for every $a \in [-2, 1]$.
Therefore, the infimum of the right hand side of \eqref{eq:quasi} over
the compact interval $a \in [-2, 1]$,
call it $B$, is strictly positive. As the left hand side of \eqref{eq:quasi}
is bounded from above by $\|f_2\|_\infty \le \|f\|_\infty$, it follows that 
relation \eqref{eq:quasi} holds for $a \in [-2,1]$ provided
we multiply the right hand side by the constant $\|f\|_\infty/B$.
%for every $\delta \in (0,\infty)$ there exists $C = C_\delta > 0$ large enough such that 
Summarizing, for all $a \in \R$
\begin{equation} \label{eq:quasi2}
	\sup_{x \in [a,a+1)} f_2(x) \le
	D \int_{a - 1}^{a + 2} \overline g_1(|w|/2)  \, \dd w \,, \quad \
	\text{where} \ \ \ D \,:=\, 2 \, \|f\|_\infty \,
	\max\bigg\{1, \frac{\|f\|_\infty}{B} \bigg\} \,.
\end{equation}

We have thus shown that \eqref{eq:hyy} holds true for 
$f_2$, with $h(\cdot) = D\, \overline g_1(|\cdot|/2)$.
Applying iteratively Lemma~\ref{th:basic}, for every $n \in\N$ we obtain
\begin{equation} \label{eq:eheh}
	\sup_{x \in [a,a+1)} f_{2^n}(x) \le 
	\int_{a-1}^{a+2} \overline h_n(x) \, \dd x
	\,, \qquad \forall a \in \R\,,
\end{equation}
where recalling \eqref{eq:quasi2} we set
\begin{equation} \label{eq:defhn}
\begin{split}
	\overline h_1(x) & \,:=\, D \, g_1(|x|/2) \,, \\
	\overline h_n(x) & \,:=\, \big( \Phi_{2^{n-1}}(\overline h_{n-1}) \big)(x) \,=\,
	\big( (\Phi_{2^{n-1}} \circ \Phi_{2^{n-2}}
	\circ \ldots \circ \Phi_{2})(\overline h_1) \big)(x) \,, \quad \ \forall n \ge 2\,.
\end{split}
\end{equation}

\subsection{Conclusion}

Let us observe that, for every $\epsilon > 0$,
\begin{equation*}
	c_{\epsilon} := \sup_{a \in \R} \bigg(
	\frac{\sup_{z \in [a-1,a+2]} (1+|z|^{\epsilon })}
	{\inf_{z \in [a-1,a+2]} (1+|z|^{\epsilon })} \bigg) < \infty \,,
\end{equation*}
and we can write $(1+|x|^\epsilon) \le c_\epsilon (1+|x'|^\epsilon)$ for all $x, x' \in \R$
with $|x-x'| \le 3$.
Applying \eqref{eq:eheh}, it follows that
we can estimate the left hand side of \eqref{eq:aimm} for $k=2^n$ as follows:
\begin{equation*}
\begin{split}
	& \sum_{m \in \Z} \sup_{z \in [m, (m+1))}
	\big( (1+|z|^{\epsilon }) \,  f_{2^n}(z) \big) \le 
	\sum_{m \in \Z} c_{\epsilon}\, (1+|m|^{\epsilon }) \,
	\sup_{z \in [m, (m+1))} 	f_{2^n}(z) \\
	& \qquad \,\le\, \sum_{m \in \Z} c_{\epsilon}\, (1+|m|^{\epsilon })
	\int_{(m-1) }^{(m+2)} \overline h_n(w)  \, \dd w
	\,\le\, \sum_{m \in \Z} c_{\epsilon}^2
	\int_{(m-1) }^{(m+2)} (1+|w|^{\epsilon }) \, \overline h_n(w)  \, \dd w \\
	& \qquad \,=\, 3\, c_{\epsilon}^2 \int_\R (1+|w|^{\epsilon }) \, \overline h_n(w)  \, \dd w \,.
\end{split}
\end{equation*}
Therefore to prove \eqref{eq:aimm} it suffices to show that
there exists $n \in \N$ large enough such that
\begin{equation} \label{eq:aimm2}
	\int_\R (1+|w|^{\epsilon}) \, \overline h_n(w)  \, \dd w \,<\, \infty \,,
\end{equation}
where $\overline h_n(\cdot)$ is defined in \eqref{eq:defhn}.
To this purpose, we show that the maps $\Phi_{n}$ are regularizing.

\begin{lemma} \label{th:Phireg}
If $h \in L^\infty \cap L^q$, for some $q \in (1,\infty)$, then
$\Phi_{n}(h) \in L^p$ for every $p \in ((1-\epsilon)q, \infty)
\cap [1,\infty)$ and for every $n\in\N$.
\end{lemma}

\begin{proof}
Recalling the definition \eqref{eq:defPhi} of the operator $\Phi_n$,
it follows by Jensen's inequality that for
every $h \in L^\infty$ and $x \in \R$
\begin{equation} \label{eq:Jen}
	|\Phi_{n} (h)(x)|^p
	\,\le\, 2^p \,	\int_{|z| > (|x| - 3)/2} \dd z \, f_{n}(z) \, |h(x-z)|^p \,,
	\qquad \forall p \ge 1 \,,
\end{equation}
hence
\begin{equation*}
	\|\Phi_{n} (h)\|_p^p \,=\, \int_\R \dd x \, |\Phi_{n} (h)(x)|^p
	\,\le\, 2^p \, \int_\R \dd z \, f_{n}(z) 
	\int_{|x| < 2|z| + 3} \dd x \, |h(x-z)|^p \,.
\end{equation*}
Under the change of variables $x \to w := x-z$,
the domain $\{|x|< 2|z| + 3\}$ becomes 
$\{-(2|z| + 3)-z < w < (2|z| + 3) - z\} \subseteq
\{-3(|z| + 1) < w < 3(|z| + 1)\}$. Therefore enlarging the domain of integration
and recalling that $g_n(t) := \int_{|z|>t} \dd z \, f_n(z)$ we obtain
\begin{equation*}
\begin{split}
	\|\Phi_{n}(h)\|_p^p & \,\le\, 2^p
	\, \int_\R \dd z \, f_{n}(z) 
	\int_{|w| < 3(|z| + 1)} \dd w \, |h(w)|^p \\
	& \,=\, 2^p \, \int_\R \dd w \, |h(w)|^p
	\int_{|z| > |w|/3 - 1} \dd z \,  f_{n}(z)
	\,=\, 2^p \, \int_\R \dd w \, |h(w)|^p
	\, g_{n} ( |w|/3 - 1 ) \,,
\end{split}
\end{equation*}
where we recall that $g_n(t) := g_n(0) = 1$ for $t < 0$.
For every $\gamma \in (0,1)$, by H\"older's inequality we then obtain
\begin{equation} \label{eq:bel}
\begin{split}
	\|\Phi_{n}(h)\|_p^p & 
	\,\le\, 2^p \, \left( \int_\R \dd w \, |h(w)|^{p/(1-\gamma)} \right)^{1-\gamma}
	\left( \int_\R \dd w \, g_{n} (|w|/3 - 1)^{1/\gamma} \right)^{\gamma} \,.
\end{split}
\end{equation}
Looking back at \eqref{eq:integrabl}, we see that the second integral
in the right hand side is finite for all $\gamma \in (0,\epsilon)$.
Now observe that if $h \in L^\infty \cap L^q$
then $h \in L^{q'}$ for all $q' \ge q$, therefore
the first integral in the right hand side is finite
whenever $p/(1-\gamma) \ge q$. Summarizing, we have shown that
$\|\Phi_{n}(h)\|_p < \infty$, that is $\Phi_{n}(h) \in L^p$,
for every $p \in [1,\infty)$ (recall relation \eqref{eq:Jen})
such that $p/(1-\gamma) \ge q$ for some $\gamma \in (0,\epsilon)$, i.e.,
for every $p \in [1,\infty) \cap ((1-\epsilon)q, \infty)$.
\end{proof}

We are almost done. Recall that we need to show that \eqref{eq:aimm2}
holds true for $n$ large enough, where
$\overline h_n(\cdot)$ is defined in \eqref{eq:defhn}.
By \eqref{eq:ass'}, we know that $\| \overline h_1 \|_q < \infty$ for
$q = 2/\epsilon$. Applying iteratively Lemma~\ref{th:Phireg}, it follows that
$\| \overline h_n \|_p < \infty$ for all
$p \in ((1-\epsilon)^{n-1} q, \infty) \cap [1,\infty)$. By choosing $n$ large
enough we may assume henceforth that $\| \overline h_{n-1} \|_1 = \int_\R
\overline h_{n-1}(w) \, \dd w < \infty$.

By definition we have
$\overline h_{n} = \Phi_{2^{n-1}}(\overline h_{n-1})$,
therefore recalling \eqref{eq:defPhi} we can write
\begin{equation*}
\begin{split}
	\int_\R (1+|x|^{\epsilon }) \, \overline h_n(x)  \, \dd x &\,=\,
	2 \int_\R \dd x 	\int_{|z| > (|x| - 3)/2} \dd z
	\, (1+|x|^{\epsilon }) \,
	f_{2^{n-1}}(z) \, \overline h_{n-1}(x-z) \,.
\end{split}
\end{equation*}
In the domain $\{|z| > (|x| - 3)/2\}$ we have
$|x|^\epsilon \le (2|z| + 3)^\epsilon \le 2^\epsilon(2^\epsilon |z|^\epsilon + 3^\epsilon)$,
because $(a + b)^\epsilon \le (2 \max\{a,b\})^\epsilon \le 2^\epsilon (a^\epsilon + b^\epsilon)$ 
for all $a,b,\epsilon \ge 0$. Therefore
\begin{equation*}
\begin{split}
	\int_\R (1+|x|)^{\epsilon } \, \overline h_n(x)  \, \dd x & \,\le\,
	2 \int_\R \dd z \, (1+ 6^\epsilon + 4^\epsilon|z|^{\epsilon}) \, f_{2^{n-1}}(z)
	\bigg( \int_{|x| < 2|z| + 3} \dd x \, \overline h_{n-1}(x-z) \bigg) \\
	& \,\le\, 2 \, \| \overline h_{n-1} \|_1 \,
	\int_\R \dd z \, (1+ 6^\epsilon + 4^\epsilon|z|^{\epsilon}) \, f_{2^{n-1}}(z) 
	\,<\, \infty \,,
\end{split}
\end{equation*}
thanks to \eqref{eq:momentk}. The proof of \eqref{eq:aimm},
and hence of Theorem~\ref{th:main}, is complete.

\appendix

\section{Some technical proofs}
\label{sec:tech}

\subsection{Proof of Lemma~\ref{th:necsuf}}
\label{sec:proofnecsuf}

Observe that $\widehat f \in L^\infty$ for every $f \in L^1$.
We recall Theorem~3 in~\cite[\S XV.3]{cf:Fel2} and its Corollary,
concerning a function $g \in L^1$ and its Fourier transform $\widehat g$:
\begin{equation}\label{eq:correspo}
	\text{if $\widehat g \in L^1$, then $g \in L^\infty$ (and is continuous);}
	\quad
	\text{if $g \in L^\infty$ and if $\widehat g \ge 0$, then $\widehat g \in L^1$.}
\end{equation}
%
%\begin{itemize}
%\item if $\widehat g \in L^1$, then $g \in L^\infty$ (and is continuous);
%\item if $g \in L^\infty$ and if $\widehat g \ge 0$, then $\widehat g \in L^1$.
%\end{itemize}
%We refer to these results as the \emph{correspondence theorem}.

Assume that $\widehat f \in L^p$ for some $p \in [1,\infty)$.
Since $\widehat f \in L^\infty$, it follows that 
$\widehat f \in L^q$ for every $q \in [p,\infty]$.
Since Fourier transform turns
convolutions into products, we have $\widehat{f_{n}} =(\widehat{f})^n$ for every $n\in\N$,
therefore for all $n \ge p$ we have $\widehat{f_{n}} \in L^1$ and, by 
\eqref{eq:correspo}, $f_{n} \in L^\infty$. Therefore $f$ 
satisfies assumption \eqref{it:1} of Theorem~\ref{th:main}.

Now assume that $f_{k_0} \in L^\infty$ for some $k_0 \in \N$. If we define
$g(x) := f_{k_0}(-x)$, we have $\widehat g = (\widehat{f_{k_0}})^* =
((\widehat{f})^{k_0})^*$,
where we denote by $a^*$ the complex conjugate of $a\in \C$.
It follows that $\widehat{f_{k_0} * g} =
\widehat{f_{k_0}} * \widehat{g} = |\widehat{f}|^{2 k_0} \ge 0$.
Note that $f_{k_0} * g \in L^\infty$, because both $f_{k_0}$
and $g$ are in $L^1 \cap L^\infty$, therefore, by \eqref{eq:correspo},
$\widehat{f_{k_0} * g} = |\widehat{f}|^{2 k_0} \in L^1$,
that is $\widehat{f} \in L^{2 k_0}$.

\subsection{Proof of Lemma~\ref{th:suf}}
\label{sec:proofsuf}

Let us denote by $\|\cdot\|_{p}$ the norm in $L^p$, i.e.,
\begin{equation*}
	\|h\|_p \,:=\, \bigg( \int_\R |h(w)|^p \, \dd w \bigg)^{1/p}
	\quad \text{if} \ p \in [1,\infty)\,, \qquad
	\|h\|_\infty \,:=\, \mathrm{ess\,sup}_{x\in\R} |h(x)| \,.
\end{equation*}
Recalling \eqref{eq:conv}, by H\"older's inequality
\begin{equation*}
\begin{split}
	|f_2(x)| & \,\le\,  \bigg( \int_\R |f(x-y)| \, |f(y)|^p \, \dd y \bigg)^{\frac{1}{p}}
	\, \|f\|_1^{\frac{p-1}{p}}  \\
	& \,\le\, \bigg( \int_\R |f(x-y)|^p \, |f(y)|^p \, \dd y \bigg)^{\frac{1}{p^2}}
	\, \|f\|_1^{\frac{p-1}{p}} \, \|f\|_p^{\frac{p-1}{p}} \,,
\end{split}
\end{equation*}
therefore $f_2 \in L^1 \cap L^{p^2}$.
By iteration, $f_{2^k} \in L^1 \cap L^{p^{2^k}}$ for every $k\in\N$.
If $k$ is such that $p^{2^k} \ge 2$, it follows that $f_{2^k} \in L^2$. 
Since Fourier transform is an isometry on $L^2$, it follows
that $\widehat {f_{2^k}} = (\widehat {f})^{2^k} \in L^2$, 
that is $\widehat f \in L^{2^{k+1}}$.
The conclusion follows by Lemma~\ref{th:necsuf}.

%\subsection{Proof of Lemma~\ref{th:reformulation}}
%\label{sec:proofreformulation}

\subsection{Proof of relation \eqref{eq:basto}}
\label{sec:proofbasto}

We recall that $g \ge 0$ by assumption.
Arguing as in \cite[\S1.8]{cf:Tsi},
for fixed $x,x' \in \R$ and $\delta, \delta' > 0$, let $A$ denote
the set of $(m,m') \in \Z^2$ such that
\begin{equation*}
	[m'\delta' - x', (m'+1)\delta' - x') \,\cap\, [m\delta-x, (m+1)\delta-x) \,\ne\, \emptyset \,.
\end{equation*}
For every $m'\in\Z$, there are at most $(\delta'/\delta) + 2$ values of $m \in \Z$
for which $(m,m') \in A$, hence
\begin{equation*}
\begin{split}
	& \sum_{m\in\Z} \, \sup_{z \in [m\delta, (m+1)\delta)} g(z-x) 
	\,\le\, \sum_{m\in\Z} \, \max_{m' \in \Z: \, (m,m') \in A}
	\sup_{z \in [m'\delta', (m'+1)\delta')} g(z-x') \\
	& \qquad \,\le\, \sum_{(m,m') \in A} 
	\sup_{z \in [m'\delta', (m'+1)\delta')} g(z-x')
	\,\le\, \bigg(\frac{\delta'}{\delta} + 2\bigg) 
	\sum_{m'\in\Z} \, \sup_{z \in [m'\delta', (m'+1)\delta')} g(z - x')  \,.
\end{split}
\end{equation*}
Now denote by $B$ the set of 
$(m,m') \in \Z^2$ such that
\begin{equation*}
	[m\delta-x, (m+1)\delta-x) \,\subseteq\, [m'\delta' - x', (m'+1)\delta' - x') \,.
\end{equation*}
Plainly, for every $m$ there is \emph{at most} one value of $m'$ such that $(m,m') \in B$,
while for every $m'$ there are \emph{at least} $(\delta'/\delta) - 2$ values of $m \in \Z$
for which $(m,m') \in B$. Therefore
\begin{equation*}
\begin{split}
	\sum_{m\in\Z} \, \inf_{z \in [m\delta, (m+1)\delta)} g(z-x) 
	& \,\ge\, \sum_{m\in\Z} \, \sum_{m'\in\Z} \, \ind_{\{(m,m') \in B\}}
	\inf_{z \in [m\delta, (m+1)\delta)} g(z-x)  \\
	& \,\ge\, \sum_{m'\in\Z} \, \inf_{z \in [m'\delta', (m'+1)\delta')} g(z-x')
	\Bigg( \sum_{m'\in\Z} \, \ind_{\{(m,m') \in B\}} \Bigg) \\
	& \,\ge\, \bigg( \frac{\delta'}{\delta} - 2 \bigg)
	\sum_{m'\in\Z} \, \inf_{z \in [m'\delta', (m'+1)\delta')} g(z-x')  \,.
\end{split}
\end{equation*}
The previous relations shown that, for all $x,x' \in \R$ and $\delta, \delta' > 0$,
\begin{equation} \label{eq:basto2}
	S^g_{\delta}(x) \,\le\, \bigg( 1 + 2 \frac{\delta}{\delta'} \bigg) S^g_{\delta'}(x') \,,
	\qquad
	s^g_{\delta}(x) \,\ge\, \bigg( 1 - 2 \frac{\delta}{\delta'} \bigg) s^g_{\delta'}(x') \,,
\end{equation}
which implies in particular \eqref{eq:basto}.
%Recall that $\lim_{\delta' \downarrow 0} S^g_{\delta'}(0)
%= \lim_{\delta' \downarrow 0} s^g_{\delta'}(0) \in (0,\infty)$,
%because $g$ is d.R.i..
%Taking $x'=0$ in \eqref{eq:basto}, letting $\delta \downarrow 0$ and then 
%$\delta' \downarrow 0$, it follows that
%\begin{equation}
%	\lim_{\delta \downarrow 0} \,
%	\sup_{x \in \R} \big( S^g_\delta(x) - s^g_\delta(x) \big) \,=\, 0 \,.
%\end{equation}
%In particular, we have shown that $g$ is d.R.i. if and only if $S^g_\delta(x)$ and $s^g_\delta(x)$
%converge to the same finite limit as $\delta \downarrow 0$ \emph{uniformly}
%over $x\in\R$.

%Now observe that
%\begin{equation*}
%	S^{g*h}_\delta(0) \,=\, 
%	\sum_{m\in\Z} \delta \, \bigg(  \sup_{z \in [m\delta, (m+1)\delta)} (g*h)(z) \bigg)
%	\,\le\, \int_\R \sum_{m\in\Z} \delta \, \bigg(
%	\sup_{z \in [m\delta, (m+1)\delta)} g(z-y) \bigg) \, h(y) \, \dd y \,,
%\end{equation*}

\subsection{Bootstrapping relation \eqref{eq:aimm}}
\label{sec:bootaimm}

Let us set $g_k(x) := (1+|x|^\epsilon) |f_k(x)|$, so that
the left hand side of \eqref{eq:aimm} can be expressed as $S^{g_k}_1(0)$
(recall \eqref{eq:SsR}).
Since $(a+b)^\epsilon \le 2^\epsilon (a^\epsilon + b^\epsilon)$
for $a,b,\epsilon \ge 0$, recalling \eqref{eq:conv} we can write
\begin{align*}
	S^{g_{k+1}}_1(0) & \,\le\, 2^\epsilon \, \int_\R 
	|f(y)| \,  \sum_{m \in \Z} \bigg( \sup_{z \in [m, (m+1))}
	(1+|z-y|^{\epsilon} + |y|^\epsilon) \, |f_{k}(z-y)| \bigg) \, \dd y \\
	& \,=\, 2^\epsilon \bigg\{ \bigg( \int_\R |f(y)| \,
	S^{g_{k}}_1(z-y) \, \dd y \bigg)
	\,+\, \bigg( \int_\R |y|^\epsilon \, |f(y)| \,
	S^{|f_{k}|}_1(z-y) \, \dd y \bigg) \bigg\} \\
	& \,\le\, 3 \cdot 2^\epsilon \bigg\{ \bigg( \int_\R |f(y)| \, \dd y \bigg)
	\,+\, \bigg( \int_\R |y|^\epsilon \, |f(y)| \, \dd y \bigg) \bigg\}
	\, S^{g_k}_1(0) \,,
\end{align*}
where the last inequality follows from \eqref{eq:basto} with
$x = z-y$, $x' = 0$ and $\delta = \delta' = 1$.
The term in brackets is finite, by the assumptions of Theorem~\ref{th:main}, hence
$S^{g_{k+1}}_1(0) < \infty$ if $S^{g_k}_1(0) < \infty$. This shows that
if \eqref{eq:aimm} holds for $k = \overline{k}$, then it holds for all $k \ge \overline{k}$.

\end{document}